\documentclass[11pt,reqno,a4paper]{amsart}
 \usepackage{amsgen, amstext,amsbsy,amsopn, amsthm, amsfonts,amssymb,amscd,amsmath,euscript,enumerate,url,verbatim,calc,tikz}
\usepackage{hyperref}
\usepackage{MnSymbol}
\usepackage{pgfkeys}
\usepackage{mathtools}
\usepackage{etex}
 \RequirePackage{etex}
\usepackage[left=1.2in,right=1.2in,bottom=1.5in]{geometry}

\usepackage{thmtools, thm-restate}
\usepackage{pst-node}

\usepackage{mathrsfs}
\usepackage{tikz-cd}

\usepackage{tikz}
\usepackage{tikz-network}
\usepackage{eqnarray,amsmath}

\def\multiset#1#2{\ensuremath{\left(\kern-.3em\left(\genfrac{}{}{0pt}{}{#1}{#2}\right)\kern-.3em\right)}}

\usetikzlibrary{arrows}

 \usepackage{latexsym}
 \usepackage{graphics}
 \usepackage{color}
\usepackage{lastpage}
\usepackage{fancyhdr}
\usepackage{multirow}
\usepackage[T1]{fontenc}
\usepackage[utf8]{inputenc}
\allowdisplaybreaks
\usepackage{graphicx}
\graphicspath{ {F:/IMAGES/} }

\makeatletter
\def\oversortoftilde#1{\mathop{\vbox{\m@th\ialign{##\crcr\noalign{\kern3\p@}%
      \sortoftildefill\crcr\noalign{\kern3\p@\nointerlineskip}%
      $\hfil\displaystyle{#1}\hfil$\crcr}}}\limits}

\def\sortoftildefill{$\m@th \setbox\z@\hbox{$\braceld$}%
  \braceld\leaders\vrule \@height\ht\z@ \@depth\z@\hfill\braceru$}

\makeatother

\newcommand{\proset}{\,\mathrel{\lower 4pt\hbox{$\scriptscriptstyle/$}
\mkern -14mu\subseteq }\,} %for proper subset

 \newtheorem{theorem}{Theorem}[section]
 \newtheorem{corollary}[theorem]{Corollary}
 \newtheorem{lemma}[theorem]{Lemma}
 \newtheorem{proposition}[theorem]{Proposition}

\usepackage{amsmath}

\usepackage{hyperref}

 \theoremstyle{definition}

 \newtheorem{definition}[theorem]{Definition}

 \newtheorem{example}[theorem]{Example}

\title[A comparison of the v-number of a monomial ideal and its integral closure]{A comparison of the v-number of a monomial ideal and its integral closure}

\author{Prativa Biswas, Mousumi Mandal and Partha Phukan}

\thanks{AMS Classification 2020:  13A15, 13B22, 13F55, 05C25, 05E40}
\address{Department of Mathematics, Indian Institute of Technology Kharagpur, 721302, India}\email{prativabiswassnts@kgpian.iitkgp.ac.in}
\address{Department of Mathematics, Indian Institute of Technology Kharagpur, 721302, India}\email{mousumi@maths.iitkgp.ac.in}
\address{Department of Mathematics, Indian Institute of Technology Kharagpur, 721302, India}\email{p.partha.24@kgpian.iitkgp.ac.in}

\begin{document}

\maketitle

\begin{abstract}
Let $I$ be a monomial ideal in a standard graded polynomial ring and let $\overline{I}$ denote its integral closure. We study the relationship between $\mathrm{v}(I)$ and $\mathrm{v}(\overline{I})$. We prove that
\(
\mathrm{v}(\overline{I}) \leq \mathrm{v}(I)
\)
for monomial ideals in two variables, for equigenerated monomial ideals in three variables and for several special classes of monomial ideals, while providing examples showing that this inequality does not hold in general. For the edge ideal $I(G)$ of a connected graph $G$, we show that
\(
\mathrm{v}(I(G)^k)=\mathrm{v}(\overline{I(G)^k})
=
2k-1
\)
for all $k \geq 1+|E(G)|$. Moreover, when $G$ is disconnected, we prove that
\(
\mathrm{v}(\overline{I(G)^k})\leq\mathrm{v}({I(G)^k})
\)
for all sufficiently large $k$.
\end{abstract}

\section{Introduction}

\noindent Let $R=K[x_1,\ldots,x_n]$ be a standard graded polynomial ring over a field
$K$, and let $I\subseteq R$ be a homogeneous ideal. The $\mathrm{v}$-number
of $I$ is defined by
\[
\mathrm{v}(I)
=
\min\left\{
d\geq 0 :
\begin{array}{l}
\text{there exist } f\in R_d \text{ and }
\mathfrak{p}\in\operatorname{Ass}(I)\\
\text{such that } (I:f)=\mathfrak{p}
\end{array}
\right\}.
\]
Thus, the $\mathrm{v}$-number records the least degree of a homogeneous
element whose colon ideal with $I$ is an associated prime. This invariant
has been studied extensively for graded ideals, particularly for monomial
ideals and ideals arising from graphs.

In this paper, we study the behaviour of the $\mathrm{v}$-number under
integral closure. Even for monomial ideals, determining the integral
closure explicitly can be difficult. However, the integral closure of a
monomial ideal admits a useful geometric description in terms of its
Newton polyhedron. More precisely, if $I\subseteq R$ is a monomial ideal
and
\(
E(I)=\{\mathbf{a}\in\mathbb{N}^n : x^{\mathbf{a}}\in I\}
\)
denotes its exponent set, \(\operatorname{conv}(E(I))\) is the convex hull of \(E(I)\) in  \(\mathbb{R}^{n}\) then the Newton polyhedron of $I$ is defined by
\[
\operatorname{NP}(I)
=
\operatorname{conv}(E(I))
=
\operatorname{conv}
\{\mathbf{a}\in\mathbb{N}^n : x^{\mathbf{a}}\in I\}.
\]
Moreover,
\[
x^{\mathbf{a}}\in\overline{I}
\quad\Longleftrightarrow\quad
\mathbf{a}\in\operatorname{NP}(I)\cap\mathbb{N}^n.
\]
We denote by $\alpha(I)$ the initial degree of $I$, defined by
\(
\alpha(I)
=
\min\{\deg(f)\mid 0\neq f\in I , \text{ is homogeneous}\}
\). Let \(\mathcal{G}(I)\) be the unique minimal set of monomial generators of \(I\). We denote by \(V(I)\) the minimal subset of \(E(\mathcal{G}(I))\) such that
\(
\operatorname{NP}(I)=\operatorname{conv}(V(I))+\mathbb{R}_{\geq 0}^{n},
\)
and call \(V(I)\) the set of vertices, or equivalently, the set of extreme points of
\(\operatorname{NP}(I)\). By \cite[ Lemma~2.4]{Hoa22}, for every
\(\mathbf{a}=(a_1,\ldots,a_n)\in V(I)\), the monomial
\(x_1^{a_1}\cdots x_n^{a_n}\) is a minimal generator of
\(\overline{I}\). This connection between integral closure and convex geometry makes it
natural to compare $\mathrm{v}(I)$ and $\mathrm{v}(\overline{I})$.
In this direction, Vanmathi.~A and P.~Sarkar \cite{VanmathiSarkar25}
established comparison results for the $\mathrm{v}$-number of complete
intersection monomial ideals and equigenerated irreducible monomial ideals
with that of their integral closures.

A related comparison problem has been considered for
Castelnuovo-Mumford regularity. In \cite{KuronyaPintye}, A. K\"uronya and N. Pintye conjectured that
taking integral closure does not increase regularity, that is,
\(
\operatorname{reg}(\overline{I})
\leq
\operatorname{reg}(I).
\)
The inequality is known for several classes of monomial ideals in \cite{Javadekar26},\cite{KumarKumar21},\cite{MinhVu22}, although
a counterexample to the conjecture has recently been obtained by S. Misra in \cite{Misra26}. Motivated
by this problem, we ask whether the analogous inequality
\(
\mathrm{v}(\overline{I})\leq \mathrm{v}(I)
\)
holds for monomial ideals. Our first main result gives an affirmative answer for monomial ideals in
two variables. More precisely, in Theorem \ref{2variablesproof} for every monomial ideal
$I\subseteq K[x,y]$, we prove that 
\[
\mathrm{v}(\overline{I})\leq \mathrm{v}(I).
\]
The proof uses the explicit description of the $\mathrm{v}$-number of
monomial ideals in two variables. In Theorem~\ref{thm:main-bound}, we proved that if \(I \subseteq K[x,y,z]\) is an equigenerated monomial ideal in degree \(d\), then
\[
\mathrm{v}(\overline{I^k}) \le \mathrm{v}(I^k)
\]
for all \(k \ge 1\). However, this inequality does not hold in general. We
provide explicit examples in three variables, including an
$\mathfrak{m}$-primary monomial ideal and a stable monomial ideal, as well as examples in four variables including equigenerated monomial ideals for which
\(
\mathrm{v}(I)<\mathrm{v}(\overline{I}).
\)

 Although the inequality fails for arbitrary monomial ideals, we establish
positive results for several important classes. In particular, we prove \(
\mathrm{v}(\overline{I})\leq \mathrm{v}(I)
\) for equigenerated $\mathfrak{m}$-primary ideals,
monomial ideals generated by pure powers of variables, and equigenerated
stable monomial ideals. We then focus on powers of edge ideals. Let $G$ be a simple connected
graph and let $I(G)$ denote its edge ideal. In Theorem \ref{prop:v-stability} we prove that
\[
\mathrm{v}\!\left(\overline{I(G)^k}\right)
=
\mathrm{v}\!\left(I(G)^k\right)
=
2k-1
\]
for every
\(
k\geq 1+|E(G)|.
\)
Thus, although $I(G)^k$ and $\overline{I(G)^k}$ may be different ideals,
their $\mathrm{v}$-numbers eventually agree, and the stage of equality is
given by an explicit bound depending only on the number of edges of $G$. The disconnected case behaves differently. We give an example showing
that the inequality can be strict. Nevertheless, in Theorem \ref{v-numberofdisconnectedgraphs} we prove that, if $G$ is disconnected
and has at least one non-bipartite connected component, then for
all sufficiently large $k$,
\[
\mathrm{v}\!\left(\overline{I(G)^k}\right)
\leq
\mathrm{v}\!\left(I(G)^k\right).
\]
More precisely, if $b$ denotes the number of bipartite connected
components and $c(G)$ the total number of connected components, then
\[
2k-1
\leq
\mathrm{v}\!\left(\overline{I(G)^k}\right)
\leq
2k+b-1
\leq
\mathrm{v}\!\left(I(G)^k\right)
=
2k+c(G)-2
\]
for all sufficiently large $k$.

\section{Preliminaries}

\noindent This section revisits fundamental concepts, terminology and results that will be used throughout the article. For unexplained terminology, we refer the interested reader to standard texts  \cite{HerzogHibi11} \cite{HunekeSwanson06}.

\begin{comment}

Let \(R = K[x_1, \dots, x_n]\) be a standard graded polynomial ring over a field \(K\), and let \(I \subseteq R\) be a homogeneous ideal. A key invariant associated with \(I\) is its \(\mathrm{v}\)-number. Although the \(\mathrm{v}\)-number of \(I\) has been thoroughly studied, extending such investigations to the integral closure of \(I\) poses substantially greater challenges. The integral closure of \(I\) often displays much richer algebraic behaviour. Even when the generators of a monomial ideal are explicitly known, describing the generators of its integral closure remains a difficult task, both computationally and theoretically.

Let \(I \subseteq R\) be a monomial ideal. For \(\mathbf{a} = (a_1, \dots, a_n) \in \mathbb{N}^n\), let \(x^{\mathbf{a}} = x_1^{a_1} \cdots x_n^{a_n}\) denote the corresponding monomial in \(R\). The Newton polyhedron of \(I\) is defined as
\[
\operatorname{NP}(I) = \operatorname{conv}\{\mathbf{a} \in \mathbb{N}^n \mid x^{\mathbf{a}} \in I\},
\]
i.e., the convex hull of all exponent vectors of monomials in \(I\). Moreover, for \(\mathbf{a} \in \mathbb{N}^n\), one has \(x^{\mathbf{a}} \in \overline{I}\) if and only if \(\mathbf{a} \in \operatorname{NP}(I)\), where \(\overline{I}\) denotes the integral closure of \(I\).

\end{comment}

\begin{definition}
Let \(I,J\subseteq R\) be two ideals. Then 
\[
(I:J)\coloneqq \{f\in R\mid fg\in I \text{ for all } g\in J\}
\]
is an ideal of \(R\) known as the colon ideal of \(I\) with respect to \(J\). For \(f\in R\) we write \((I:f)\coloneqq (I:(f))\). By \cite[Proposition 1.2.2]{HerzogHibi11}, for a monomial ideal \(I\subseteq R\) and a monomial \(f\in R\), we have
\[
(I:f) = \left(\frac{u}{\gcd(u,f)}\mid u\in \mathcal{G}(I)\right).
\]
\end{definition}

\begin{definition}

Let \(I\) be an ideal in a ring \(R\). An element \(x \in R\) is called integral over \(I\) if there exists an integer \(m \geq 1\) and elements \(a_i \in I^i\) for \(1 \leq i \leq m\) such that
\[
x^{m} + a_{1}x^{m - 1} + \dots +a_{m - 1}x + a_{m} = 0.
\]
The integral closure of \(I\), denoted by \(\overline{I}\), is an ideal in \(R\), which is the set of all elements in \(R\) that are integral over \(I\) .

\end{definition}

\begin{definition}
    A (simple) \textit{graph} ${G}$ is defined by a pair $(V({G}), E({G}))$, where $V({G})$ is a finite set, called the \textit{vertex set} of $G$ and  $E({G})$, called the \textit{edge set} of $G$, is the family of subsets of $V(G)$ such that they are pairwise incomparable with respect to inclusion and the cardinality of each element of $E({G})$ is two. 
\end{definition}

Suppose ${G}$ is a graph with $V({G})=\{x_1,x_2,\ldots,x_n\}$ and $E({G})$ as the edge set. We consider each vertex $x_i$ as a variable of the polynomial ring $R=K[x_1,x_2,\ldots,x_n]$ in $n$ variables over a field $K$. We can assign a monomial ideal to $G$ called the \textit{edge ideal} of $G$, denoted by $I(G)$, in the following way:
    $$I({G}):=\langle\prod_{x_i\in e}x_i~|~e\in E({G})\rangle.$$

\begin{definition}
    A \textit{cycle} is a simple graph with the same number of vertices and edges whose vertices can be placed around a circle so that two vertices are adjacent if they appear consecutively along the circle. A cycle with $n$ vertices is denoted by $C_n$.
\end{definition}
\begin{definition}
    A simple graph $G$ is said to be a \textit{bipartite} graph if its vertex set $V(G)$ can be partitioned into two non-empty subsets $X$ and $Y$ such that every edge in $G$ has one endpoint in $X$ and the other endpoint in $Y$. Otherwise it is a \textit{non-bipartite} graph
\end{definition}

\begin{lemma}[{\cite[Theorem 5.9]{SimisVasconcelosVillarreal94}}]
\label{lem:normal-torsion-free-bipartite}
Let \(G\) be a graph. Then \(I(G)\) is normally torsion-free if and only if \(G\) is bipartite.
\end{lemma}

\begin{definition}\label{def:reg-pd-depth}
Let $\mathbf{F}$ be a minimal graded free resolution of $R / I$ as an $R$-module:
\[
\mathbf{F}:\;0\to \bigoplus_{j}R(-j)^{\beta_{p,j}}\to \cdots \to \bigoplus_{j}R(-j)^{\beta_{1,j}}\to R\to R/I\to 0,
\]
where $I$ is a homogeneous ideal of $R$.
 The \emph{Castelnuovo--Mumford regularity} of $R/I$, denoted by $\operatorname{reg}(R/I)$, is defined by
    \[
    \operatorname{reg}(R/I)=\max\{j-i\mid \beta_{i,j}\neq 0\}.
    \]
Let $M$ be a finitely generated graded $R$-module. Regularity can also be defined via the vanishing of local cohomology modules with respect to the unique maximal homogeneous ideal
\(
\mathfrak{m} = (x_1,\ldots,x_n).
\)
For every $i\geq 0$, define
\[
a_i(M)
=
\max\{n:H_{\mathfrak{m}}^i(M)_n\neq 0\}.
\]
Then
\[
\operatorname{reg}(M)
=
\max\{a_i(M)+i:i\geq 0\}.
\]

\end{definition}

\begin{definition}

 For a monomial \( x^{A} \in R \), write \( \mu(x^{A}) = \max\{ i : x_i \mid x^{A} \} \). A monomial ideal \( I \subseteq R \) is called stable if for every \( x^{A} \in I \) and every \( i < \mu(x^{A}) \) we have \( x_i \cdot \frac{x^{A}}{x_{\mu(x^{A})}} \in I \).

\end{definition}

 For a monomial ideal \(I \subseteq K[x,y]\), we denote by \(\mathcal{G}(I)\) its unique minimal set of monomial generators. This set can be expressed as \(\mathcal{G}(I) = \{x^{a_{1}}y^{b_{1}}, x^{a_{2}}y^{b_{2}}, \dots, x^{a_{m}}y^{b_{m}}\}\), where the sequences satisfy \(a_{1} > a_{2} > \dots > a_{m} \ge 0\) and \(0 \le b_{1} < b_{2} < \dots < b_{m}\). Conversely, any pair of sequences \((\mathbf{a}, \mathbf{b})\) meeting these conditions determines a monomial ideal through the generating set above. Consequently, there is a bijection between monomial ideals of \(K[x, y]\) and pairs of sequences \((\mathbf{a}, \mathbf{b})\) as specified, with \(\mathbf{a} : a_{1} > a_{2} > \dots > a_{m} \ge 0\) and \(\mathbf{b} : 0 \le b_{1} < b_{2} < \dots < b_{m}\). In the sequel, we denote such an ideal by \(I = (x^{a_{1}}y^{b_{1}}, x^{a_{2}}y^{b_{2}}, \dots, x^{a_{m}}y^{b_{m}})\).

\begin{comment}

\noindent In \cite[Proposition 5.5]{FicarraSgroi23}, Ficarra and Sgroi proved that if \(I \subseteq K[x,y]\) is a non-principal \(\mathfrak{m}\)-primary monomial ideal, then 
\[
\mathrm{v}(I) = \min\{a_i + b_{i+1} - 2 : 1 \le i \le m-1\},
\]
where \(b_1 = 0\) and \(a_m = 0\).

\end{comment}

\begin{theorem}[{\cite[Theorem 5.6]{FicarraSgroi23}}]\label{thm:FicarraSgroi23}

\noindent Let $I \subseteq  K[x, y] $ be a monomial ideal. Then
\[
\mathrm{v}(I) = \left\{
\begin{array}{ll}
\min \{a_i + b_{i+1} - 2 : 1 \leq i \leq m-1\}, \quad \text{if }   b_1 = 0 \text{ and } a_m = 0, \\[4pt]
\min \{a_1 + b_1 - 1,\; a_i + b_{i+1} - 2 : 1 \leq i \leq m-1\}, \quad \text {if } b_1 \neq 0 \text{ and } a_m = 0, \\[4pt]
\min \{a_m + b_m - 1,\; a_i + b_{i+1} - 2 : 1 \leq i \leq m-1\}, \quad \text {if } b_1 = 0 \text{ and } a_m \neq 0, \\[4pt]
\min \{a_1 + b_1 - 1,\; a_m + b_m - 1,\; a_i + b_{i+1} - 2 : 1 \leq i \leq m-1\},  \quad \text{otherwise}.
\end{array}
\right.
\]

\end{theorem}

\section{v-number of a monomial ideal and its integral closure}

\noindent In this section, we compare the $\mathrm{v}$-number of a monomial ideal with that of its integral closure. We first prove that
\(
\mathrm{v}(\overline{I})\leq \mathrm{v}(I)
\)
for every monomial ideal $I\subseteq K[x,y]$. We then provide explicit examples showing that this inequality does not hold in general, even for $\mathfrak{m}$-primary monomial ideals, stable monomial ideals and equigenerated monomial ideals. We also establish the inequality
\(
\mathrm{v}(\overline{I})\leq \mathrm{v}(I),
\)
and in some cases equality, for several important classes of monomial ideals.

\begin{proposition}\label{{v(I^bar)<=v(I)}}
Let \(I\subseteq K[x,y]\) be a non-principal \(\mathfrak m\)-primary
monomial ideal. Then
\[
\mathrm{v}(\overline{I})\leq \mathrm{v}(I),
\]
where \(\overline{I}\) denotes the integral closure of \(I\).
\end{proposition}

\begin{proof}
Let
\(I=(x^{a_1},x^{a_2}y^{b_2},\ldots,
x^{a_{m-1}}y^{b_{m-1}},y^{b_m})\), where
\(a_1>a_2>\cdots>a_{m-1}>a_{m}=0\) and
\(b_1=0<b_2<\cdots<b_m\) . 
By Theorem \ref{thm:FicarraSgroi23},
\begin{equation}\label{eq:vI}
\mathrm{v}(I)
=\min\{a_i+b_{i+1}-2:1\leq i\leq m-1\}.
\end{equation}

\noindent\textbf{Case 1:}
Suppose that every minimal generator of \(I\) is also a minimal
generator of \(\overline{I}\). For \(1\leq i\leq m-1\), let
\(w_{i,1},\ldots,w_{i,r_i}\) be the minimal generators of
\(\overline{I}\) lying strictly between
\(u_i=x^{a_i}y^{b_i}\) and
\(u_{i+1}=x^{a_{i+1}}y^{b_{i+1}}\). Put
\(w_{i,0}=u_i\), \(w_{i,r_i+1}=u_{i+1}\), and write
\(w_{i,j}=x^{c_{i,j}}y^{d_{i,j}}\). Then
\(c_{i,0}=a_i>c_{i,1}>\cdots>c_{i,r_i+1}=a_{i+1}\) and
\(d_{i,0}=b_i<d_{i,1}<\cdots<d_{i,r_i+1}=b_{i+1}\).

\noindent For a fixed \(i\), the consecutive pairs of minimal generators of
\(\overline{I}\) in this block are
\(w_{i,j},w_{i,j+1}\), where \(0\leq j\leq r_i\). Hence, by Theorem
\ref{thm:FicarraSgroi23}, the corresponding terms in the
computation of \(\mathrm{v}(\overline{I})\) are
\(c_{i,j}+d_{i,j+1}-2\). Since \(c_{i,j}\leq a_i\) and
\(d_{i,j+1}\leq b_{i+1}\), we have
\[
c_{i,j}+d_{i,j+1}-2\leq a_i+b_{i+1}-2
\quad\text{for all }0\leq j\leq r_i.
\]
Let
\(\mu_i=\min\{c_{i,j}+d_{i,j+1}-2:0\leq j\leq r_i\}\).
Then \(\mu_i\leq a_i+b_{i+1}-2\). Since every consecutive pair of
minimal generators of \(\overline{I}\) belongs to one of these
blocks, we have
\(\mathrm{v}(\overline{I})=\min\{\mu_i:1\leq i\leq m-1\}\).
Therefore,
\[
\mathrm{v}(\overline{I})
\leq
\min\{a_i+b_{i+1}-2:1\leq i\leq m-1\}
=\mathrm{v}(I).
\]

\noindent\textbf{Case 2:}
Suppose that
\(\mathcal{G}(I)\cap \mathcal{G}(\overline{I})=\{x^{a_1},y^{b_m}\}\) 
and for \(i = 2, \ldots, m-1\), we have
\(u_i'=x^{a_i'}y^{b_i'}\in \mathcal{G}(\overline{I})\) such that
\(u_i'\mid u_i\). Set
\(u_1'=u_1\) and \(u_m'=u_m\). Since \(u_i'\mid u_i\), we have
\(a_i'\leq a_i\) and \(b_i'\leq b_i\) for every \(i\). For a fixed \(i\), let
\(w_{i,1},\ldots,w_{i,r_i}\) be the minimal generators of
\(\overline{I}\) lying strictly between \(u_i'\) and \(u_{i+1}'\).
Put \(w_{i,0}=u_i'\), \(w_{i,r_i+1}=u_{i+1}'\), and write
\(w_{i,j}=x^{c_{i,j}}y^{d_{i,j}}\). Then
\(c_{i,j}\leq a_i'\) and \(d_{i,j+1}\leq b_{i+1}'\) for
\(0\leq j\leq r_i\). Thus, every term arising from two consecutive
minimal generators in this block satisfies
\[
c_{i,j}+d_{i,j+1}-2
\leq a_i'+b_{i+1}'-2
\leq a_i+b_{i+1}-2.
\]
Let
\(\mu_i=\min\{c_{i,j}+d_{i,j+1}-2:0\leq j\leq r_i\}\).
Then
\(\mu_i\leq a_i'+b_{i+1}'-2\leq a_i+b_{i+1}-2\).
Since every consecutive pair of minimal generators of
\(\overline{I}\) belongs to one of these blocks, we obtain
\[
\begin{aligned}
\mathrm{v}(\overline{I})
&=\min\{\mu_i:1\leq i\leq m-1\}\\
&\leq
\min\{a_i'+b_{i+1}'-2:1\leq i\leq m-1\}\\
&\leq
\min\{a_i+b_{i+1}-2:1\leq i\leq m-1\}
=\mathrm{v}(I).
\end{aligned}
\]

\noindent For the remaining cases, some minimal generators of \(I\) remain
minimal generators of \(\overline{I}\), while the others are replaced
by minimal generators of \(\overline{I}\) dividing them. Applying the
blockwise argument of Case~1 together with the inequalities
\(a_i'\leq a_i\) and \(b_i'\leq b_i\) from Case~2 gives
\(\mathrm{v}(\overline{I})\leq\mathrm{v}(I)\).
\end{proof}

\begin{theorem}\label{2variablesproof}
Let \(I \subseteq K[x,y]\) be a monomial ideal and let \(\overline{I}\) denote its integral closure. Then
\[
\mathrm{v}(\overline{I}) \le \mathrm{v}(I).
\]
\end{theorem}

\begin{proof}
The proof follows from Theorem \ref{thm:FicarraSgroi23}, applying the arguments of Proposition \ref{{v(I^bar)<=v(I)}}, and the observation that, if
\[
\mathcal{G}(I)=\{x^{a_1}y^{b_1},\ldots,x^{a_m}y^{b_m}\},
\]
with \(a_1>\cdots>a_m\geq 0\) and \(0\leq b_1<\cdots<b_m\), then \(x^{a_1}y^{b_1}\) and \(x^{a_m}y^{b_m} \in \mathcal{G}(\overline{I})\).
\end{proof}

\begin{comment}
    
In \cite{KuronyaPintye}, Küronya -Pintye conjectured that integral closure does not increase Castelnuovo-Mumford regularity, i.e., \(\operatorname{reg}(\overline{I}) \le \operatorname{reg}(I)\), where \(I\) is a homogeneous ideal of \(R\). The inequality is known for several classes of monomial ideals in \cite{Javadekar26},\cite{KumarKumar21}. Later in \cite{Misra26}, Misra produces a counterexample to the conjecture.
\end{comment}

A natural question that arises is whether \(\mathrm{v}(\overline{I}) \le \mathrm{v}(I)\) holds for a monomial ideal \(I\) of \(R\). In the next example, we provide a counterexample to this statement for \(\mathfrak{m}\)-primary homogeneous ideals.

\begin{example}\label{counterexampletomprimary}
Let $I=(x^2,y^2,z^5,xyz)\subseteq K[x,y,z]$. Then 
\(
\overline I=(x^2,xy,y^2,xz^3,yz^3,z^5)
\)
and $\mathrm{v}(I)=2\le \mathrm{v}(\overline I)=3$.

\begin{proof}
Let $(a,b,c)\in \mathrm{NP}(I)$. Then the following inequality holds:
\[
5a+5b+2c\ge 10.
\]
Since this inequality is satisfied by the exponent vectors of the minimal monomial generators of $I$, it remains valid on their convex hull and is therefore preserved under the addition of $\mathbb R_{\ge 0}^3$. Consequently, it holds for all points in $\mathrm{NP}(I)$.

Let $J=(x^2,xy,y^2,xz^3,yz^3,z^5)$. The vectors $(1,1,0)$, $(1,0,3)$, $(0,1,3)$ satisfy $5a+5b+2c\ge 10$, so they lie in $\mathrm{NP}(I)$; hence $J\subseteq \overline I$.
Conversely, take $x^a y^b z^c\in \overline I$. Then $(a,b,c)\in \mathrm{NP}(I)$, so
\[
a,b,c\ge 0,\qquad 5a+5b+2c\ge 10. \tag{1}
\]
If $a\ge 2$, $b\ge 2$, or $a,b\ge 1$, then $x^2$, $y^2$, or $xy$ divides the monomial. Otherwise $a\le 1$, $b\le 1$, and not both positive. If $a=b=0$, (1) gives $c\ge 5$, so $z^5\mid x^a y^b z^c$. If $a=0,b=1$, then $c\ge 3$, so $yz^3\mid x^a y^b z^c$. If $a=1,b=0$, then $c\ge 3$, so $xz^3\mid x^a y^b z^c$. Thus every monomial in $\overline I$ lies in $J$. Hence $\overline I\subseteq J$, proving equality.
Since $I:xy=(x,y,z)$ and $\overline I:yz^2=(x,y,z)$, we have $\mathrm{v}(I)=2\le \mathrm{v}(\overline I)=3$.  One can use Macaulay2 \cite{GraysonStillman} to verify the above example.

\end{proof}
\end{example}

\noindent We now investigate classes of monomial ideals \(I\) for which
\(
\mathrm{v}(\overline{I})\leq \mathrm{v}(I).
\)

\begin{theorem}\label{equigeneratedmprimary}
Let $I$ be an equigenerated $\mathfrak  {m}$-primary monomial ideal in \(R\) with $\alpha(I) = d$. Then $\mathrm{v}(\overline{I}) = d - 1 \leq \mathrm{v}({I})$. Moreover, for every \(k \ge 1\),
\[
\mathrm{v}(\overline{I^k}) = kd - 1 \le \mathrm{v}(I^k).
\]
\end{theorem}

\begin{proof}
Since $I$ is an equigenerated $\mathfrak{m}$-primary ideal, we have $x_i^d \in I$ for all $i = 1, 2, \ldots, n$. Consequently, $d e_i \in \operatorname{E}(I) \subseteq \operatorname{NP}(I)$ for each $i$. 
For $i = 2, \ldots, n$, consider the convex combination
\[
\frac{d-1}{d}(d e_1) + \frac{1}{d}(d e_i) = (d-1)e_1 + e_i.
\]
Since $\operatorname{NP}(I)$ is convex, it follows that $(d-1)e_1 + e_i \in \operatorname{NP}(I)$ for all $i = 2, \ldots, n$. Moreover, $(d-1)e_1 + e_1 = d e_1 \in \operatorname{NP}(I)$ as well. Hence $x_1^{d-1} x_i \in \overline{I}$ for all $i = 1, 2, \ldots, n$.
Therefore, $(\overline{I} : x_1^{d-1}) = \mathfrak{m}$. This yields $\mathrm{v}(\overline{I}) \le d - 1$. On the other hand, by \cite[Lemma 3.1]{BiswasMandalSaha26} we have $\mathrm{v}(\overline{I}) \ge \alpha(\overline{I}) - c(\overline{I}) = d - 1$. Thus $\mathrm{v}(\overline{I}) = d - 1 \leq \mathrm{v}({I})$. Since \(I^k\) is an \(\mathfrak{m}\)-primary equigenerated monomial ideal for every \(k \ge 1\), the result follows.
\end{proof}

\begin{proposition}
Let \(I = (x_1^{a_1}, \dots, x_r^{a_r})\) be a monomial ideal generated by pure powers of variables in \(R \) with \(r \le n\). Then 
\[
\mathrm{v}(\overline{I}) \le \mathrm{v}(I).
\]
\end{proposition}

\begin{proof}
The ideal \(I\) is \(\mathfrak{p}\)-primary where \(\mathfrak{p} = (x_1, \dots, x_r)\). Consider the monomial 
\(f = x_1^{a_1 - 1} x_2^{a_2 - 1} \cdots x_r^{a_r - 1}\). 
One checks that \((I : f) = (x_1, \dots, x_r) = \mathfrak{p}\), so \(\mathrm{v}(I) \le \deg f\). 
If there existed a monomial \(h\) with \(\deg h < \deg f\) then \((I : h) \neq \mathfrak{p}\) . Thus \(f\) corresponds to the \(\mathrm{v}\)-number, i.e., \(\mathrm{v}(I) = \deg f\). Now let \(g = x_1^{b_1} \cdots x_r^{b_r}\) be a monomial corresponding to \(\mathrm{v}(\overline{I})\), meaning \((\overline{I} : g) \in \operatorname{Ass}(\overline{I})\) and \(\deg g = \mathrm{v}(\overline{I})\). 
If \(b_i \ge a_i\) for some \(i\), then \(g\) would belong to \(I \subseteq \overline{I}\). Hence \(b_i < a_i\) for all \(i = 1, \dots, r\). Consequently,
\(
\deg g = \sum_{i=1}^r b_i \le \sum_{i=1}^r (a_i - 1) = \deg f = \mathrm{v}(I).
\)
Thus \(\mathrm{v}(\overline{I}) \le \mathrm{v}(I)\).
\end{proof}

\begin{theorem}
    Let $I$ be a monomial ideal in \(R\) with $\mathrm{v}(I)=\alpha(I)-1$. If $\overline{I}$ has no embedded associated primes , then $\mathrm{v}(\overline{I})=\alpha(I)-1.$ 
\end{theorem}

\begin{proof}
    Let $f$ be a monomial corresponding to $\mathrm{v}(I).$ Then $(I:f)=\mathfrak p$ with $\deg f=\alpha(I)-1$ for some $\mathfrak p\in \operatorname{Ass}(I).$ It is clear that $(I:f)\subseteq(\overline{I}:f).$ If $(I:f)=(\overline{I}:f),$ then we are done. Now if possible let $(I:f)\subseteq(\overline{I}:f),$ then there exists $\mathfrak q\in\operatorname{Ass}(\overline{I}:f)$ such that $((\overline{I}:f):g)=\mathfrak q.$ This gives $\mathfrak q\in\operatorname{Ass}(\overline{I}).$ Since $\operatorname{Min}\operatorname{Ass}(\overline{I})=\operatorname{Min}\operatorname{Ass}(I)$ and $\mathfrak p\in \operatorname{Ass}(I),$ there exists $\mathfrak{p'}\in\operatorname{Min}\operatorname{Ass}(\overline{I})$ such that $\mathfrak p'\subseteq\mathfrak{p}.$ Therefore $\mathfrak q$ and $\mathfrak p'\in\operatorname{Ass}(\overline{I}),$ where $\mathfrak p'\subseteq \mathfrak p\subset \mathfrak q.$ This is a contradiction to the fact that $\overline{I}$ has no embedded prime. Thus $(I:f)=(\overline{I}:f)=\mathfrak p,$ which gives $\mathrm{v}(\overline{I})=\deg f=\alpha(I)-1.$
\end{proof}

\begin{proposition}\label{stableequigenerated}
Let \( I \) be a stable monomial ideal in  \( R \) that is equigenerated in degree \( d \), and let \( \overline{I} \) denote its integral closure. Then \( \overline{I} \) is also stable and equigenerated in degree \( d \).
\end{proposition}

\begin{proof}
Since \( I \) is stable, \cite[Proposition 6.5]{Javadekar26} implies that its integral closure \( \overline{I} \) is stable as well. Now let \( u \) be a minimal generator of \( \overline{I} \). Since \( I \) is stable and equigenerated in degree \( d \), then by \cite[Corollary 7.2.3]{{HerzogHibi11}} its \( \operatorname{reg}(I) = d \). By the integral closure criterion for monomial ideals, there exist an integer \( r \ge 1 \) and monomials \( f_1,\dots,f_r \in I \) such that \( u^r = f_1 \cdots f_r \). Each \( f_i \) has degree at least \( d \), so \( r\deg(u) = \sum_{i=1}^r \deg(f_i) \ge r d \), giving \( \deg(u) \ge d \). On the other hand, \cite[Proposition 6.6]{Javadekar26} yields \( \operatorname{reg}(\overline{I}) \le \operatorname{reg}(I) = d \), and since \( \overline{I} \) is stable, its regularity equals the maximum degree of a minimal generator. Hence \( \deg(u) \le d \). Combining the inequalities forces \( \deg(u) = d \), so \( \overline{I} \) is equigenerated in degree \( d \).
\end{proof}

\begin{corollary}
Let \(I\) be a stable monomial ideal in \( R\) that is equigenerated in degree \(d\). Then, for every \(k \ge 1\),
\[
\mathrm{v}(\overline{I^k}) = \mathrm{v}(I^k) = kd - 1
\]
\end{corollary}

\begin{proof}

It follows from Proposition  \ref{stableequigenerated} and \cite[Proposition 4.10. and 4.16.]{MandalPhukan26}.

\end{proof}

\noindent In general, $\mathrm{v}(\overline I)$ need not be less than or equal to $\mathrm{v}(I)$, even for stable monomial ideals. The following is a counterexample to this.

\begin{example}
Let 
\(
I=(x^2,\;xy^2,\;y^3,\;y^2z^3)\subseteq K[x,y,z]
\) is a stable monomial ideal. Using Macaulay2 \cite{GraysonStillman} we have 
\(
\overline I=(x^2,\;xy^2,\;y^3,\;xyz^2,\;y^2z^3).
\)
Thus 
\(
\mathrm{v}(I)=2\le \mathrm{v}(\overline I)=3.
\)
\end{example}

\begin{proposition}\label{equigeneratedht(1)orht(2)}
Let \(I \subseteq K[x,y,z]\) be an equigenerated monomial ideal in degree \(d\). If \(\operatorname{ht}(I) = 1\) or \(3\), then
\[
\mathrm{v}(\overline{I}) \le \mathrm{v}(I).
\]
\end{proposition}

\begin{proof}
Suppose first that \(\operatorname{ht}(I) = 1\), the ideal has a minimal associated prime of height one, which must be one of \((x)\), \((y)\), or \((z)\). Therefore, \(u = \gcd(\mathcal{G}(I)) \neq 1\). Write \(I = u I_1\), where \(\gcd(\mathcal{G}(I_1)) = 1\). Choose a variable, say \(x\), dividing \(u\), and choose \(g \in \mathcal{G}(I_1)\) such that \(x \nmid g\). Set
\(
f = (\frac{u}{x}) g.
\)
Then \(\deg f = d - 1\) and \(x \in (I : f)\). Write \(u = x^a u'\), where \(a \ge 1\) and \(x \nmid u'\). Since \(x \nmid g\), we have
\(
f = x^{a-1} u' g.
\)
Let \(h \in (I : f)\). Then \(h f \in I = u I_1\), so the exponent of \(x\) in \(h \) is at least \(1\). This implies \(h \in (x)\). Hence \((I : f) = (x)\). Now, since \(\overline{I} = \overline{u I_1} = u \overline{I_1}\), we have
\(
(\overline{I} : f) = (x).
\)
Therefore, using \cite[Lemma 3.1]{BiswasMandalSaha26}  \(\mathrm{v}(\overline{I}) =\mathrm{v}(I) = d - 1 \). For the case \(\operatorname{ht}(I) = 3\), the result follows from Proposition \ref{equigeneratedmprimary}. This completes the proof.
\end{proof}

\begin{proposition}
\label{equigeneratedht(2)}
Let \(I \subseteq K[x,y,z]\) be an equigenerated monomial ideal in degree \(d\) with \(\operatorname{ht}(I) = 2\). Then
\[
d - 1 \le \mathrm{v}(\overline{I}) \le \min\{d,\ \mathrm{v}(I)\}.
\]
\end{proposition}

\begin{proof}
Let \(V(I)\) be the set of vertices of \(NP(I)\) and let \(\delta(I) = \max\{|v| : v \in V(I)\}\). Since \(I\) is equigenerated, we have \(\delta(I) = d\). By  \cite[Theorem 2.7]{Hoa22},
\[
\delta(I) \le \operatorname{reg}(\overline{I}) \le \delta(I) + \dim R/I.
\]
Since \(\operatorname{reg}(R/\overline{I}) + 1 = \operatorname{reg}(\overline{I})\), we have
\(
\operatorname{reg}(R/\overline{I}) \le d.
\) We claim that \(\mathrm{v}(\overline{I}) \le \operatorname{reg}(R/\overline{I}) \le d\).
\textbf{Case 1:} Let \(\mathfrak{m} \in \operatorname{Ass}(\overline{I})\). Then \(H^0_{\mathfrak{m}}(R/\overline{I}) \neq 0\). By  \cite[Proposition 2.2]{FiorindoGhosh25}, we have
\[
\mathrm{v}(\overline{I}) \le a_0(R/\overline{I}) \le \operatorname{reg}(R/\overline{I}) \le d.
\]
\textbf{Case 2:} Let \(\mathfrak{m} \notin \operatorname{Ass}(\overline{I})\). Since \(\operatorname{ht}(\overline{I}) = 2\), its minimal primes are generated by two variables. Clearly all its associated primes have the same height, so \(\overline{I}\) is unmixed. Moreover, \(\dim(R/\overline{I}) = 1\) and the associated primes are generated by linear forms. Therefore \(\overline{I}\) is a Geramita ideal. Hence by  \cite[Theorem 4.10]{Cooper20},
\[
\mathrm{v}(\overline{I}) \le \operatorname{reg}(R/\overline{I}) \le d.
\]

\noindent Thus, in all cases, we have \(\mathrm{v}(\overline{I}) \le d\). Now, if \(\mathrm{v}(I) \ge d\), then \(\mathrm{v}(I) \ge \mathrm{v}(\overline{I})\). On the other hand, if \(\mathrm{v}(I) = d - 1\), we want to prove that \(\mathrm{v}(\overline{I}) = d - 1\). Indeed, choose a monomial \(g = x^a y^b z^c\) with \(a + b + c = d - 1\) such that \((I : g) = \mathfrak{p} \in \operatorname{Ass}(I)\). Since \(\operatorname{ht}(I) = 2\), either \(\mathfrak{p} = \mathfrak{m}\) or \(\mathfrak{p}\) is generated by two variables. If \(\mathfrak{p} = \mathfrak{m}\), then \(\mathfrak{m} \subseteq (\overline{I} : g)\). It follows that \((\overline{I} : g) = \mathfrak{m}\), hence \(\mathrm{v}(\overline{I}) = d - 1\).
\noindent Now suppose \((I : g) = (x, y)\). Then \(xg\) and \(yg\) are minimal generators of \(I\). Let
\[
A = (a+1, b, c), \qquad B = (a, b+1, c)
\]
be the exponent vectors of \(xg\) and \(yg\), respectively. Suppose \(z^r \in (\overline{I} : g)\) for some \(r \ge 1\). We claim that \(r = 1\). Let \(Q = (a, b, c + r)\), the exponent vector of \(g z^r \in \overline{I}\). Since \(Q \in NP(I) = \operatorname{conv}(\mathcal{G}(I)) + \mathbb{R}_{\ge 0}^3\), there exists a point \(D = (u, v, w) \in \operatorname{conv}(\mathcal{G}(I))\) and \((\alpha, \beta, \gamma) \in \mathbb{R}_{\ge 0}^3\) such that
\[
Q = D + (\alpha, \beta, \gamma).
\]
It follows that
\(
u \le a, \quad v \le b, \quad w \le c + r.
\)
Let \(u = a - \lambda\) and \(v = b - \mu\), where \(\lambda, \mu \ge 0\). Since \(D \in \operatorname{conv}(\mathcal{G}(I))\), we have \(u + v + w = d\). Therefore,
\(
w = d - u - v = d - a - b + \lambda + \mu = c + 1 + \lambda + \mu.
\)
Hence,
\[
(a, b, c+1) = \frac{1}{1 + \lambda + \mu} D + \frac{\lambda}{1 + \lambda + \mu} A + \frac{\mu}{1 + \lambda + \mu} B.
\]
Since \(D, A, B \in \operatorname{conv}(\mathcal{G}(I))\), it follows that \((a, b, c+1) \in \operatorname{conv}(\mathcal{G}(I)) \subseteq NP(I)\). Thus \(x^a y^b z^{c+1} = gz \in \overline{I}\). This implies \(z \in (\overline{I} : g)\). Therefore, \((\overline{I} : g) = (x, y, z) = \mathfrak{m}\), and hence \(\mathrm{v}(\overline{I}) = d - 1\). This completes the proof.

\end{proof}

\begin{theorem}
\label{thm:main-bound}
Let \(I \subseteq K[x,y,z]\) be an equigenerated monomial ideal in degree \(d\). Then
\[
\mathrm{v}(\overline{I^k}) \le \mathrm{v}(I^k)
\]
for all \(k \ge 1\).
\end{theorem}

\begin{proof}
The result follows from the fact that \(I^k\) is again an equigenerated monomial ideal, combined with Proposition \ref{equigeneratedht(1)orht(2)} and Proposition \ref{equigeneratedht(2)}.
\end{proof}

\noindent In general, the statement is not true for equigenerated ideals in four variables.

\begin{example}
Let
\(
I = (w^4,\ yzw^2,\ xyw^2,\ x^2z^2,\ x^2yw) \subseteq K[x,y,z,w]
\)
be a stable monomial ideal. Using Macaulay2 \cite{GraysonStillman}, we compute
\[
\overline I = (w^4,\ yzw^2,\ xyw^2,\ x^2z^2,\ x^2yw,\ xzw^2,\ xyz^2w).
\]
Thus,
\[
\mathrm{v}(I) = 3 \le \mathrm{v}(\overline I) = 4.
\]
This shows that the inequality \(\mathrm{v}(\overline I) \le \mathrm{v}(I)\) need not hold for equigenerated monomial ideals in four or more variables.
\end{example}

\begin{proposition}
\label{prop:v-stability-squarefree}
Let \(I\) be a squarefree monomial ideal such that \(\mathrm{v}(I) = \alpha(I) - 1\). Then
\[
\mathrm{v}(\overline{I^k}) = \mathrm{v}(I^k) = k\alpha(I) - 1
\]
for all \(k \ge 1\).
\end{proposition}

\begin{proof}
Let \(f\) be a squarefree monomial such that
\(
\deg f = \alpha(I) - 1 \quad \text{and} \quad (I : f) = \mathfrak{p}.
\)

\noindent Then \(x_{i_j}\) does not divide \(f\) for all \(x_{i_j} \in \mathcal{G}(\mathfrak{p})\). Let \(\mathfrak{p} = (x_{i_1}, \ldots, x_{i_s})\).
\noindent We claim that
\[
(\overline{I^k} : x_{i_1}^{k-1} f^k) = \mathfrak{p}.
\]
Indeed, since
\(
x_{i_j} x_{i_1}^{k-1} f^k = (x_{i_1}^{k-1} f^{k-1})(x_{i_j} f) \in I^k \subseteq \overline{I^k},
\)
it follows that \(\mathfrak{p} \subseteq (\overline{I^k} : x_{i_1}^{k-1} f^k)\).

\noindent In order to prove the other inclusion, let \(u = x_1^{a_1} \cdots x_n^{a_n}\) be a minimal generator of \(I^k\). Define
\(
w_{\mathfrak{p}}(u) = {\sum_{x_j \in \mathfrak{p}}} a_j.
\)
Then \(w_{\mathfrak{p}}(u) \ge k\) for all \(u \in \mathcal{G}(I^k)\).
All exponent vectors of monomials in \(I^k\) lie in 
\(
\left\{ \mathbf{a} : \sum_{x_j \in \mathfrak{p}} a_j \ge k \right\}.
\)
Hence, every point in \(NP(I^k)\) lies in
\(
\left\{ \mathbf{a} : \sum_{x_j \in \mathfrak{p}} a_j \ge k \right\}.
\)
Note that, \(w_{\mathfrak{p}}(f) = 0\), so
\(
w_{\mathfrak{p}}(x_{i_1}^{k-1} f^k) = k - 1.
\)
For any monomial \(v \notin \mathfrak{p}\), we have \(w_{\mathfrak{p}}(v) = 0\), and therefore
\(
w_{\mathfrak{p}}(v x_{i_1}^{k-1} f^k) = k - 1.
\)
Thus, \(v x_{i_1}^{k-1} f^k \notin \overline{I^k}\), which implies
\(
v \notin (\overline{I^k} : x_{i_1}^{k-1} f^k).
\)
Hence,
\[
(\overline{I^k} : x_{i_1}^{k-1} f^k) = \mathfrak{p}.
\]
Therefore,
\[
\mathrm{v}(\overline{I^k}) \le \deg(x_{i_1}^{k-1} f^k) = (k-1) + \deg f^k = (k-1) + k(\alpha(I) - 1) = k\alpha(I) - 1.
\]
Since \(\mathrm{v}(I^k) \ge k\alpha(I) - 1\) for all \(k \ge 1\), we obtain
\[
\mathrm{v}(\overline{I^k}) = \mathrm{v}(I^k) = k\alpha(I) - 1.
\]
This completes the proof.
\end{proof}

\section{v-number of powers of edge ideal and its integral closure}

\noindent In this section, we study the $\mathrm{v}$-numbers of powers of edge ideals and their integral closures. We prove that if $G$ is a simple connected graph, then
\(
\mathrm{v}\!\left(\overline{I(G)^k}\right)
=
\mathrm{v}\!\left(I(G)^k\right)
\)
for all $k\geq 1+|E(G)|$. We also consider the case when $G$ is disconnected and show that
\(
\mathrm{v}\!\left(\overline{I(G)^k}\right)
\leq
\mathrm{v}\!\left(I(G)^k\right)
\)
for all sufficiently large $k$.\\

\noindent Suppose that \(G\) is a simple non-bipartite connected graph. Let \(V(G)\), \(E(G)\), and \(\varepsilon_0(G)\) denote its vertex set, edge set, and number of leaf edges of \(G\), respectively. It has a spanning tree \(T\) which has exactly \(|V(G)|-1\) edges.  Since \(G\) is non-bipartite, it contains an odd cycle. Then
\(
|E(G)| \ge |V(G)|.
\)
Let \(2m-1\) be the maximum length of odd cycles of \(G\), where \(m \ge 2\). Then
\(
|V(G)| - \varepsilon_0(G) - m + 1 \le 1 + |E(G)|.
\)

\begin{theorem}\label{prop:v-stability}
Let $G$ be a simple connected graph, and let $I(G)$ be its edge ideal. Then
\[
\mathrm{v}(I(G)^k)=\mathrm{v}(\overline{I(G)^k})=2k-1
\]
for all
\(
k\ge 1+|E(G)|\,
\).

\end{theorem}

\begin{proof}
If $G$ is a bipartite graph, then by \cite[Proposition 2.1]{SimisVasconcelosVillarreal98}, $I(G)$ is normal.

\noindent If $G$ is a non-bipartite graph, and let $2m-1$ be the maximum length of odd cycles of $G$, then by \cite[Lemma 3.1]{MauTrung21},  for any
\(
k\ge |V(G)|-\varepsilon_0(G)-m+1,
\)
there is a monomial $f$ of degree $2k-1$ such that
\(
\mathfrak m = (I(G)^k:f)
\).
Therefore
\(
\mathrm{v}(I(G)^k)\le 2k-1.
\)
Again using \cite[Lemma 3.1]{BiswasMandalSaha26} , we have
\(
\mathrm{v}(I(G)^k)\ge \alpha(I(G)^k)-1=2k-1.
\)
Moreover, $f\notin \overline{I(G)^k}$, since $\deg f < \alpha(\overline{I(G)^k})=2k$. Therefore
\(
\mathfrak m = (\overline{I(G)^k}:f).
\)
Thus
\(
\mathrm{v}(I(G)^k)=2k-1.
\)
Hence, using \cite[Theorem 5.2]{BiswasMandalSaha26}  for a simple connected graph $G$,
\[
\mathrm{v}(I(G)^k)=\mathrm{v}(\overline{I(G)^k})=2k-1
\]
for all
\(
k\ge 1+|E(G)|\,
\).
\end{proof}

\begin{corollary}

Let \(I(G)\) be the edge ideal of a simple connected graph \(G\). Suppose that \(J\) is a monomial ideal such that
\(
I(G)^k\subseteq J\subseteq \overline{I(G)^k}
\)
for some \(k\geq 1+|E(G)|\). Then
\[
\mathrm{v}(J)=2k-1.
\]
\end{corollary}

\begin{proof}
Since
\(
\alpha(\overline{I(G)^k})\leq \alpha(J)\leq \alpha(I(G)^k)
\)
and
\(
\alpha(\overline{I(G)^k})=\alpha(I(G)^k)=2k,
\)
we have \(\alpha(J)=2k\).
The rest follows by the same argument as in the preceding Theorem \ref{prop:v-stability}, which completes the proof.
\end{proof}

\noindent The above theorem fails when we do not consider the graph \(G\) to be connected. The following example highlights, when \(G\) is a disconnected graph, then
\(
\mathrm{v}(\overline{I(G)^k}) < \mathrm{v}(I(G)^k)
\)
for some \(k\).

\begin{example}
Let \(G\) be the union of two disjoint odd cycles of length \(3\). Then
\[
\mathrm{v}(\overline{I(G)^7}) < \mathrm{v}(I(G)^7).
\]
\end{example}

\begin{proof}

Throughout this example, let \(G = C_3 \cup C_3\) denote the union of two disjoint odd cycles of length \(3\), and let \(R = K[x_1,x_2,x_3,y_1,y_2,y_3]\) be the polynomial ring over a field \(K\) in six variables. The edge ideal of \(G\) decomposes as \(I(G) = I_1 + I_2\), where
\[
I_1 = (x_1x_2,\ x_1x_3,\ x_2x_3), \qquad
I_2 = (y_1y_2,\ y_1y_3,\ y_2y_3).
\]
Furthermore, set \(X = x_1x_2x_3\), \(Y = y_1y_2y_3\), and \(M = XY\).

\noindent By \cite[Theorem 10.5.12]{Vil26}, for each pair \(C_i, C_j\) of induced disjoint odd cycles, we set \(m_{C_i,C_j} = x_{C_i}x_{C_j}\). If \(P\) denotes the set of all pairs of disjoint odd cycles in \(G\), then
\[
\overline{R[It]} = R[It,\ m_{C_i,C_j}t^{m_{i,j}/2} \mid (C_i,C_j) \in P],
\]
where \(m_{i,j} = \deg m_{C_i,C_j}\). In our case, we have \(\overline{R[It]} = R[It,\ Mt^3]\), and consequently
\[
[\overline{R[It]}]_7 = (I^7 + MI^4)t^7.
\]

\noindent Since \(I_1\) and \(I_2\) are squarefree monomial ideals of degree \(2\), with \(\mathrm{v}(I_1) = \alpha(I_1) - 1\) and \(\mathrm{v}(I_2) = \alpha(I_2) - 1\), it follows from \cite[Corollary 4.2]{BiswasMandalSaha26} that
\[
\mathrm{v}(I_1^q) = 2q - 1 \quad\text{and}\quad \mathrm{v}(I_2^q) = 2q - 1
\]
for all \(q \ge 1\). Moreover, \(\mathrm{v}_{\mathrm{stab}}(I_1) = \mathrm{v}_{\mathrm{stab}}(I_2) = 1\), and hence  by \cite[Theorem 5.2]{FicarraMarques25}, 
\[
\mathrm{v}((I_1 + I_2)^k) = 2k + (2 + 2 - 2 - 2) = 2k.
\]
Thus, for \(k = 7\), we obtain \(\mathrm{v}(I(G)^7) = 14\).

\noindent Now  \( \overline{I^7} = I^7 + MI^4\), and consider the monomial \(f = M y_3 (x_1x_2)^3\). A direct computation gives
\[
(\overline{I^7} : f) = (x_1, x_2, x_3, y_1, y_2) = \mathfrak{p} \in \operatorname{Ass}(\overline{I^7}).
\]
Therefore, \(\mathrm{v}(\overline{I^7}) = 13 < 14 = \mathrm{v}(I^7)\). This shows that the inequality \(\mathrm{v}(\overline{I(G)^k}) < \mathrm{v}(I(G)^k)\) can indeed occur for disconnected graphs.

\end{proof}

The preceding example naturally raises the following question: for a simple disconnected graph \(G\), does the inequality
\(
\mathrm{v}(\overline{I(G)^k}) \le \mathrm{v}(I(G)^k)
\)
hold for all sufficiently large integers \(k\)? In order to address this,  we prove the next theorem.

\begin{definition}
\label{def:do}
\cite[Definition 3.3]{MauTrung21}
Let \(G\) be a graph with connected components \(G_1,\ldots,G_r\) such that all \(G_1,\ldots,G_r\) are non-bipartite. For each \(i=1,\ldots,r\), let \(2m_i-1\) be the maximum length of odd cycles of \(G_i\). Let \(2m-1\) be the minimum length of odd cycles of \(G\). Let
\[
d_0=
\begin{cases}

\sum_{i=1}^{r}(|V(G_i)|-\varepsilon_0(G_i)-m_i)+s+1 & \text{if } r=2s+1,  \text{for}\ s\ge 0,\\[4pt]
\sum_{i=1}^{r}(|V(G_i)|-\varepsilon_0(G_i)-m_i)+s+m & \text{if } r=2s,    \text{for}\ s\ge 1.
\end{cases}
\]
\end{definition}

\begin{theorem}\label{v-numberofdisconnectedgraphs}
Let \(G\) be a finite simple disconnected graph with no isolated vertices whose non-trivial connected components consist of bipartite graphs \(H_1,\ldots,H_b\) and non-bipartite graphs \(G_1,\ldots,G_r\), where \(r \ge 1\). Let \(c(G) = b + r\) denote the number of connected components of \(G\), and set \(I = I(G)\). Then, for all sufficiently large integers \(k\),
\[
2k - 1 \le \mathrm{v}(\overline{I(G)^k}) \le 2k + b - 1 \le \mathrm{v}(I(G)^k) = 2k + c(G) - 2.
\]
\end{theorem}

\begin{proof}
Let \(H = H_1 \cup \cdots \cup H_b\) and \(F = G_1 \cup \cdots \cup G_r\) denote the union of the bipartite and non-bipartite components of \(G\), respectively, and set \(L = I(H)\), \(J = I(F)\). The variables occurring in \(L\) and \(J\) are disjoint, and \(I(G) = L + J\). If \(b = 0\), then \(L = 0\) and \(I(G) = J\). In this case, by \cite[Lemma 3.4]{MauTrung21} and  \cite[Lemma 3.1]{BiswasMandalSaha26}, we have
\(
\mathrm{v}(\overline{I(G)^k}) = 2k - 1.
\) Now assume \(b \ge 1\). Since \(L\) is normally torsion-free, by \cite[Theorem 2.1]{MauTrung21} we have
\[
\overline{(L + J)^k} =  \sum_{i=0}^k L^i \overline{J^{\,k-i}}.
\]
Again, by \cite[Corollary 5.4]{FicarraMarques25}, \(\mathrm{v}(L^a) = 2a + b - 2\) for all sufficiently large \(a\). Choose such an \(a\) large enough, and let \(u\) be a monomial and \(\mathfrak{p} \in \operatorname{Ass}(L^a)\) such that
\[
(L^a : u) = \mathfrak{p}, \qquad \deg u = \mathrm{v}(L^a) = 2a + b - 2.
\]
We claim that \(u \in L^{a-1}\). Indeed, for any \(x \in \mathfrak{p}\), we have \(xu \in L^a\). Hence there exist edges \(e_1,\ldots,e_a \in L\) such that \(e_1\cdots e_a\) divides \(xu\), which implies \(u \in L^{a-1}\), as claimed.

Now, every connected component of \(F\) is non-bipartite. Let \(\mathfrak{n} = (y : y \in V(F))\). By the construction in the proof of \cite[Lemma 3.4]{MauTrung21}, there exists an integer \(d_0\) such that for every \(d \ge d_0\), there exists a monomial \(f_d\) satisfying
\[
\deg f_d = 2d - 1 \quad \text{and} \quad (\overline{J^d} : f_d) = \mathfrak{n}.
\]
We claim that \(f_d \in \overline{J^{d-1}}\). Indeed, suppose there is an odd number of non-bipartite components, say \(r = 2s + 1\), and write
\(
f_d = f_1 f_2 \cdots f_{2s+1},
\)
where \(f_i^2 \in I(G_i)^{2d_i - 1}\) and \(\deg f_i = 2d_i - 1\). Therefore,
\[
f_d^2 \in \prod_{i=1}^{2s+1} I(G_i)^{2d_i - 1} \subseteq J^{\sum_{i=1}^{2s+1} (2d_i - 1)},
\]
and since \(\sum_{i=1}^{2s+1} (2d_i - 1) = 2d - 1\), we obtain \(f_d^2 \in J^{2d-1}\). 
Suppose now that the number of non-bipartite components is even. Let \(r = 2s\), and write
\(
f_d = g f_1 \cdots f_{2s},
\)
where \(g\) is the product of the vertices of an odd cycle of length \(2m-1\) such that \(g^2 \in J^{2m-1}\), and \(f_i^2 \in I(G_i)^{2d_i-1}\). Then
\(
f_d^2 \in J^{2d-1}.
\)
Thus, in either case, \(f_d^2 \in J^{2d-1} \subseteq J^{2d-2} = (J^{d-1})^2\), so \(f_d \in \overline{J^{d-1}}\).
Fix \(a\) as above. Now, for sufficiently large \(k\), set \(d = k - a + 1\). We claim that
\[
(\overline{I(G)^k} : u f_d) = \mathfrak{p} + \mathfrak{n}.
\]
Indeed, we have
\[
(L^i \overline{J^{k-i}} : u f_d) = (L^i : u)(\overline{J^{k-i}} : f_d).
\]
Therefore,
\[
(\overline{I(G)^k} : u f_d) = \sum_{i=0}^{k} (L^i : u)(\overline{J^{k-i}} : f_d).
\]

\noindent \textbf{Case 1:} \(0 \le i \le a-1\). Then \((L^i : u) = R\), and \((\overline{J^{k-i}} : f_d) \subseteq \mathfrak{n}\). In particular, when \(i = a-1\), we have \(k-i = d\), so
\[
(L^{a-1} : u)(\overline{J^d} : f_d) = \mathfrak{n}.
\]

\noindent \textbf{Case 2:} \(a \le i \le k\). Since \(i \ge a\), we have \(L^i \subseteq \mathfrak{p}\), and \((\overline{J^{k-i}} : f_d) = R\). In particular, when \(i = a\), we have \(k-i = d-1\), so
\[
(L^a : u)(\overline{J^{d-1}} : f_d) = \mathfrak{p}.
\]

\noindent Hence,
\[
(\overline{I(G)^k} : u f_d) = \mathfrak{p} + \mathfrak{n}.
\]
Therefore,
\[
2k - 1 \le \mathrm{v}(\overline{I(G)^k}) \le 2k + b - 1 \le \mathrm{v}(I(G)^k) = 2k + c(G) - 2
\]
for all sufficiently large integers \(k\).
This completes the proof.

\end{proof}


\begin{thebibliography}{25}


\bibitem{VanmathiSarkar25}
Vanmathi. A and P. Sarkar,
\textit{$\operatorname{v}$-numbers of integral closure filtrations of monomial ideals},
arXiv preprint arXiv:2506.09051, 2025.


\bibitem{BiswasMandalSaha26}{P. Biswas, M. Mandal and K. Saha}, \textit{Asymptotic behaviour and stability index of v-numbers of graded ideals,} Vietnam Journal of Mathematics, 2026.


\bibitem{Cooper20} {S. M. Cooper, A. Seceleanu, S. O. Tohaneanu, M. Vaz Pinto and R. H. Villarreal}, \textit{ Generalized minimum
distance functions and algebraic invariants of Geramita ideals}, Adv. Appl. Math. 112 (2020), 101940,
34 pp.

\bibitem{FicarraSgroi23}{A. Ficarra and E. Sgroi}, \textit {Asymptotic behaviour of the v-number of homogeneous ideals}, J. Algebra \textbf{704} (2026), 273--297.


\bibitem{Hoa22}
L. T. Hoa,
\textit{Maximal generating degrees of integral closures of powers of monomial ideals},
J. Algebraic Combin. \textbf{56}(2), 279--304 (2022).



\bibitem{KuronyaPintye}{A. Küronya and N. Pintye}, \textit{Castelnuovo–Mumford regularity and log-canonical thresholds}, arXiv preprint, https://arxiv.org/abs/1312.7778, 2013.

\bibitem{Javadekar26}{O. Javadekar}, \textit {A comparison of the regularity of certain classes of monomial ideals and their integral closures},
Archiv der Mathematik 126 (2026), 351–363.


\bibitem{Misra26}{S. Misra}, \textit {A counterexample to a conjecture of Küronya and Pintye on regularity and integral closure},  arXiv preprint, https://arxiv.org/abs/2605.13879, 2026. 


\bibitem{KumarKumar21}{A. Kumar, R. Kumar}, \textit {Regularity comparison of symbolic powers, integral closure of powers and powers of edge ideals}, arXiv preprint, https://arxiv.org/abs/2108.08609, 2021.

\bibitem{FicarraMarques25} A. Ficarra and P. Macias Marques, \textit{The v-function of powers of sums of ideals}, J. Algebr. Combin. \textbf{62}, 13 (2025).


\bibitem{FiorindoGhosh25}{L. Fiorindo, D. Ghosh}, \textit{ On the asymptotic behaviour of the Vasconcelos invariant for graded modules}, Nagoya Math. J. (2025), 15 pp., published online
https://doi.org/10.1017/nmj.2024.33.



\bibitem{GraysonStillman} D. Grayson and M. Stillman, \textit{Macaulay2, A Software System for Research in Algebraic Geometry}, available at \url{https://www.unimelb-macaulay2.cloud.edu.au}.


\bibitem{GrisaldeReyesVillarreal21}{G. Grisalde, E. Reyes, R.H. Villarreal}, \textit {Induced matchings and the v-number of graded ideals}.
Mathematics, 9(22), 2021.


\bibitem{HerzogHibi11}{J. Herzog and T. Hibi}, {\textit {Monomial Ideals. Springer-Verlag London Limited, 2011}}


\bibitem{MandalPhukan26}{M. Mandal and P. Phukan},\textit { Asymptotic v-number of graded families of ideals  and the Newton-Okounkov region}, arXiv preprint, https://arxiv.org/abs/2603.08838, 2026

\bibitem{MauTrung21}
D. H. Mau and T. N. Trung,
\textit{Stability of associated primes and depth of integral closures of powers of edge ideals},
Preprint (2021), arXiv preprint, https://arxiv.org/abs/2108.01830.



\bibitem{MinhVu22} N. C. Minh and T. Vu, \textit{Integral closure of powers of edge ideals and their regularity}, J. Algebra \textbf{609} (2022), 120--144.


\bibitem{SimisVasconcelosVillarreal94}
A. Simis, W. V. Vasconcelos, and R. H. Villarreal,
\textit{On the Ideal Theory of Graphs},
J. Algebra \textbf{167} (1994), 389--416.


\bibitem{SimisVasconcelosVillarreal98}
A. Simis, W. V. Vasconcelos, and R. H. Villarreal,
\textit{The integral closure of subrings associated to graphs},
J. Algebra, 199(1):281--289, 1998.

\bibitem{HunekeSwanson06}{ I. Swanson and C. Huneke},{\textit{ Integral Closure of Ideals, Rings, and Modules}, London Mathematical Society Lecture Note Series, Vol. 336, Cambridge University Press, Cambridge, 2006.}

\bibitem{Vil26} R. H. Villarreal, \textit{Monomial Algebras}, Chapman and Hall/CRC, 3rd edition, 2026.


    
\end{thebibliography}
\end{document}